\documentclass[10pt]{article}

\usepackage{amsmath, amsthm, amssymb}
\usepackage{mathtools}
\usepackage{enumitem}
\usepackage{graphicx}

\newcommand{\iid}{i.i.d}
\newcommand{\rv}{r.v}

\newcommand{\mat}[1]{\ensuremath{{\bf #1}}}
\DeclareMathOperator{\tr}{tr}
\DeclareMathOperator{\var}{Var}

\DeclareMathOperator{\scale}{scale}
\DeclareMathOperator*{\extrema}{extrema}

\newcommand{\effrank}{\ensuremath{r_{\mathrm{eff}}}}
\newcommand{\stabrank}{\ensuremath{r_{\mathrm{stab}}}}
\newcommand{\effshape}{\ensuremath{\alpha_{\mathrm{eff}}}}

\newcommand{\ddt}{\frac{d}{dt}}
\newcommand{\prob}{\mathrm{Pr}}
\newcommand{\qlam}{\ensuremath{Q_{\asvector{\lambda}}}}
\newcommand{\qmu}{\ensuremath{Q_{\asvector{\mu}}}}
\newcommand{\arel}{\ensuremath{\mathcal{A}_{\mathrm{rel}}}}
\newcommand{\aabs}{\ensuremath{\mathcal{A}_{\mathrm{abs}}}}
\newcommand{\qrel}{\ensuremath{\mathcal{Q}_{\mathrm{rel}}}}
\newcommand{\qabs}{\ensuremath{\mathcal{Q}_{\mathrm{abs}}}}

\newcommand{\precabs}{\ensuremath{\preceq_{F}}}

\newcommand{\xupper}{\overline{x}_{\mathrm{upper}}}
\newcommand{\xlower}{\overline{x}_{\mathrm{lower}}}
\newcommand{\xhatupper}{\hat{x}_{\mathrm{upper}}}

\newcommand{\asvector}[1]{\ensuremath{\boldsymbol{#1}}}

\newcommand{\simiid}{\stackrel{\text{i.i.d.}}{\sim}}

\newcounter{parentnumber}

\newtheorem{theorem}{Theorem}

\newtheorem{lemma}{Lemma}

\newtheorem{definition}{Definition}
\newtheorem{corollary}{Corollary}
\newtheorem{conjecture}{Conjecture}

\theoremstyle{remark}

\usepackage{hyperref}

\title{Extremal bounds for Gaussian trace estimation}
\author{Eric Hallman}
\date{\today}

\begin{document}
	\maketitle

    \begin{changemargin}{1cm}{1cm}
    \begingroup
    \footnotesize
    {\bf Abstract}: This work derives extremal tail bounds for the Gaussian trace estimator applied to a real symmetric matrix. We define a partial ordering on the eigenvalues, so that when a matrix has greater spectrum under this ordering, its estimator will have worse tail bounds. This is done for two families of matrices: positive semidefinite matrices with bounded effective rank, and indefinite matrices with bounded 2-norm and fixed Frobenius norm. In each case, the tail region is defined rigorously and is constant for a given family.
    \endgroup
    \end{changemargin}

	\section{Introduction}
    Let $\mat{A}\in \mathbb{R}^{n\times n}$ be a symmetric matrix with eigenvalues $\lambda_1\geq \ldots\geq \lambda_n$, and let $\mat{z}_1,\ldots,\mat{z}_m\in \mathbb{R}^n$ be \iid~standard Gaussian vectors. The {\em Gaussian trace estimator} is given by
	\begin{equation}\label{def:G}
		\tr(\mat{A}) \approx \tr_m^G(\mat{A}) := \frac{1}{m}\sum_{j=1}^m\mat{z}_j^T\mat{A}\mat{z}_j.
	\end{equation}
    It is well known that this estimator is unbiased and has variance $\tfrac{2}{m}\|\mat{A}\|_F^2$.
 
    Often, it is useful to know how many samples $m$ are needed to produce an estimate that satisfies a given error tolerance. Cortinovis and Kressner \cite{cortinovis2021indefinite} have used concentration inequalities to derive good results on this problem, with slight improvements by Persson, Cortinovis, and Kressner in \cite{persson2022improvedvariantshutchalgorithm}.

    \begin{theorem}[\cite{cortinovis2021indefinite}, Thm.~1]\label{thm:cort_kress_abs}
        Let $\mat{A}\in \mathbb{R}^{n\times n}$ be symmetric. Then for all $\varepsilon > 0$, 
        \begin{equation*}
            \prob\left(|\tr_m^G(\mat{A}) - \tr(\mat{A})| \geq \varepsilon\right) \leq 2\exp \left( -\frac{m\varepsilon^2}{4\|\mat{A}\|_F^2 + 4\varepsilon\|\mat{A}\|_2}\right).
        \end{equation*}
    \end{theorem}
    For nonzero symmetric positive semidefinite (SPSD) matrices, this result can be turned into a relative error estimate. The trace estimator will generally be more accurate on matrices with tightly clustered eigenvalues. 
    \begin{definition}
        The {\em effective rank} of a nonzero SPSD matrix $\mat{A}$ is 
        \begin{equation*}
            \effrank(\mat{A}) := \frac{\tr(\mat{A})}{\|\mat{A}\|_2}.
        \end{equation*}
    \end{definition}

    \begin{corollary}[\cite{cortinovis2021indefinite}, Remark 2]\label{cor:cort_kress_rel}
        For nonzero SPSD $\mat{A}\in \mathbb{R}^{n\times n}$, replace $\varepsilon$ by $\varepsilon\cdot \tr(\mat{A})$ in Theorem \ref{cor:cort_kress_rel}, and note that $\|\mat{A}\|_F^2/\tr(\mat{A})^2 \leq \effrank(\mat{A})^{-1}$. For $\varepsilon > 0$, 
        \begin{equation*}
            \prob\left(|\tr_m^G(\mat{A}) - \tr(\mat{A})| \geq \varepsilon\cdot \tr(\mat{A})\right) \leq 2\exp \left( -\frac{m \varepsilon^2\cdot \effrank(\mat{A})}{4(1+\varepsilon)}\right).
        \end{equation*}
    \end{corollary}

    The goal of this paper is to tighten these bounds as much as possible using techniques from \cite{szekely2003gaussian, roosta2015gamma, roosta2015schur}. The practical benefit may be marginal, since the above bounds are already quite good. Still, I found the techniques interesting.

    \subsection{Extremal bounds}

    Words like {\em tight} and {\em optimal} can be slippery. To be precise, I mean to find {\em extremal} tail probabilities: ones that can be expressed as the supremum or infimum over some set. The error bounds of Theorem \ref{thm:cort_kress_abs} and Corollary \ref{cor:cort_kress_rel} use only certain information about $\mat{A}$; respectively, they use $(\|\mat{A}\|_2, \|\mat{A}\|_F)$ and $\effrank(\mat{A})$. It is therefore worth considering the sets of all matrices with the same summary statistics:
    \begin{align*}
        \aabs(\lambda, \phi) &:= \{\mat{A}\in \mathcal{S}_{\infty}\,:\,\|\mat{A}\|_2\leq \lambda,\, \|\mat{A}\|_F = \phi\},\quad 0 < \lambda \leq \phi, \\
        \arel(\mu) &:= \{\mat{A}\in \mathcal{S}_{\infty}^+\,:\, \|\mat{A}\|_2 \leq \tfrac{1}{\mu},\, \tr(\mat{A}) = 1\},\quad 1 \leq \mu,
    \end{align*}
    where $\mathcal{S}_\infty$ is the set of all symmetric matrices of any dimension, and $\mathcal{S}_\infty^+$ is the set of all SPSD matrices of any dimension. For $\arel(\mu)$, the parameter $\mu$ is a lower bound on the effective rank of $\mat{A}$; we can fix $\tr(\mat{A})$ without loss of generality since the relative error of $\tr_m^G(\mat{A})$ is scale invariant.

    For each of the above sets $\mathcal{A}$, we define a partial ordering:
    \begin{itemize}
        \item For $\arel(\mu)$, the partial ordering is the vector majorization order on the eigenvalues. This paper will show that $\asvector{\lambda}_{\mat{A}}\preceq \asvector{\lambda}_{\mat{B}}$ implies that $\tr_m^G(\mat{B})$ has worse upper and lower tail bounds than $\tr_m^G(\mat{A})$.
        \item For $\aabs(\lambda, \phi)$, the partial ordering is related to the majorization order but is a little more complicated. This paper will show that $\asvector{\lambda}_{\mat{A}}\preceq \asvector{\lambda}_{\mat{B}}$ implies that $\tr_m^G(\mat{B})$ has worse upper tail bounds {\em and} that $\tr_m^G(\mat{A})$ has worse lower tail bounds.
    \end{itemize}

    With these partial orderings, $\arel(\mu)$ has a maximum element and $\aabs(\lambda, \phi)$ has both maximum and minimum elements.\footnote{Strictly speaking, the maximum is unique only if we treat matrices as equivalent when they have the same nonzero eigenvalues.} 

    \subsection{The tail}
    It is not particularly useful to show that if $\asvector{\lambda}_{\mat{A}}\preceq \asvector{\lambda}_{\mat{B}}$ then one trace estimate {\em eventually} has worse tail bounds than the other\textemdash this much can already be deduced just by comparing the asymptotic behavior of the probability distributions.

    The critical result, then, is that we can define the ``tail'' regions so that their value depends {\em only on $\mathcal{A}$}. For error tolerances within these regions, the partial ordering works as advertised on all elements of $\mathcal{A}$.

    Unfortunately, it has so far been beyond my ability to prove exactly where these tail regions begin. This paper contains some pessimistic bounds, but otherwise is restricted to conjecture.

    \subsection{The extremal matrices}
    For the relative error, the matrix with the worst tail bounds is
    \begin{equation}\label{def:worst_case_A_rel}
        \mat{A}_{\mathrm{rel}}(\mu) = \frac{1}{\mu}\begin{bmatrix}
            \mat{I}_{\lfloor\mu\rfloor} & \mat{0} & \mat{0} \\
            \mat{0} & \mu - \lfloor \mu \rfloor & \mat{0} \\
            \mat{0} & \mat{0} & \mat{0}
        \end{bmatrix}.
    \end{equation}
    For the absolute error, the matrix with the worst upper tail bounds is
    \begin{equation}\label{def:worst_case_A_abs}
         \mat{A}_{\mathrm{abs}}(\lambda, \phi) = \lambda \begin{bmatrix}
             \mat{I}_{\lfloor\rho\rfloor} & \mat{0} & \mat{0} \\
            \mat{0} & \sqrt{\rho - \lfloor \rho \rfloor} & \mat{0} \\
            \mat{0} & \mat{0} & \mat{0}
        \end{bmatrix}, \quad \rho := \frac{\phi^2}{\lambda^2},
    \end{equation}
    where $\rho$ is an upper bound on the stable rank of $\mat{A}$. By symmetry, the matrix with the worst lower tail bounds is $-\mat{A}_{\mathrm{abs}}(\lambda, \phi)$.
    \begin{definition}
        The {\em stable rank} of a nonzero matrix $\mat{A}$ is
        \begin{equation*}
            \stabrank(\mat{A}) := \frac{\|\mat{A}\|_F^2}{\|\mat{A}\|_2^2}.
        \end{equation*}
    \end{definition}

    \subsubsection{Extension to Gamma random variables}
    When $\mu$ and $\rho$ are not integers, the distributions of the trace estimators for \eqref{def:worst_case_A_rel} and \eqref{def:worst_case_A_abs} are not quite as elegant as I would have liked them to be. We can, however, relax the problem by considering more general linear combinations of Gamma random variables, as opposed to only those whose distribution can be expressed as the Gaussian trace estimate of some matrix as in \eqref{def:G}.
    
    For distribution families $\qrel$ and $\qabs$ to be defined later, we will show that
    \begin{equation}\label{eqn:extrema_qrel}
        \max\, \qrel \sim Gamma\left(\frac{m\mu}{2}, \frac{m\mu}{2}\right)
    \end{equation}
    and
    \begin{equation}\label{eqn:extrema_qabs}
        \extrema\, \qabs = \pm \lambda Q,\quad Q\sim Gamma\left(\frac{m\rho}{2}, \frac{m}{2}\right),\,\rho := \frac{\phi^2}{\lambda^2}.
    \end{equation}
    When $\mu$ and $\rho$ are integers, these are exactly the distributions of the trace estimators $\tr_m^G(\mat{A}_{\mathrm{rel}})$ and $\tr_m^G(\mat{A}_{\mathrm{abs}})$ from \eqref{def:worst_case_A_rel} and \eqref{def:worst_case_A_abs}.

    \subsection{Related work}

    This paper is primarily a spiritual successor to the works \cite{szekely2003gaussian, roosta2015gamma, roosta2015schur}. Sz\'{e}kely \cite{szekely2003gaussian} gives thorough tail bounds for the relative error when $Q$ is a nonnegative sum of chi-squared \rv's with no further restrictions (the worst case is typically $Q\sim \chi_1^2$) and provides a majorization result as a corollary. Roosta-Khorasani, Sz\'{e}kely, and Ascher \cite{roosta2015gamma} extend the main result to where $Q$ is a nonnegative sum of Gamma \rv's with arbitrary shape and scale parameters. This allows them to consider the effects of using a larger sampling number $m$ for trace estimation. Finally, Roosta-Khorasani and Sz\'{e}kely \cite{roosta2015schur} generalize the worst-case bounds of \cite{roosta2015gamma} by obtaining a majorization result for nonnegative sums of Gamma \rv's. The present work's Theorem \ref{thm:main:rel}, in particular, is more or less a restatement of a result from \cite{roosta2015schur}.

    This current paper is novel in two main ways. First, it considers how the above bounds might be strengthened when the matrix in question has a large effective rank. Second, it obtains absolute error bounds for indefinite matrices \textemdash that is, when the sum of Gamma random variables is not necessarily nonnegative.

    \section{Majorization}

    In \cite{roosta2015schur}, Roosta-Khorasani and Sz\'{e}kely use the {\em majorization order} on the eigenvalues of a matrix as a measure of the ``skewness'' of its spectrum. Their general observation is that the Gaussian trace estimator performs worse when the spectrum is highly skewed. In other words: the majorization order is the partial ordering that they (and we) use to determine for which matrices the Gaussian trace estimator yields the worst relative error bounds.

    Given a nonnegative vector $\asvector{\lambda} \in \mathbb{R}^n$, let $\lambda_{[i]}$ denote its indices in sorted order, so that $\lambda_{[1]} \geq \ldots \geq \lambda_{[n]}\geq 0$. We say that \asvector{\lambda} {\em majorizes} \asvector{\mu}, denoted $\asvector{\mu} \preceq \asvector{\lambda}$, if
	\begin{subequations}
		\begin{align}
			\sum_{i=1}^k \mu_{[i]} &\leq \sum_{i=1}^k \lambda_{[i]}, \quad \forall k \leq n, \label{def:maj_dom}\\
			\sum_{i=1}^n \mu_i &= \sum_{i=1}^n \lambda_i. \label{def:maj_sum}
		\end{align}
	\end{subequations}
	Similarly, we say that \asvector{\lambda} {\em weakly majorizes} \asvector{\mu}, denoted $\asvector{\mu} \preceq_w \asvector{\lambda}$, just as long as \eqref{def:maj_dom} holds. If $\asvector{\lambda}$ and $\asvector{\mu}$ do not have the same length, pad the shorter one with zeroes as necessary.

    If $\asvector{\lambda}$ weakly majorizes $\asvector{\mu}$ but does not majorize it, then $\sum_{i=1}^n\mu_i < \sum_{i=1}^n\lambda_i$. For our proofs, it will be useful to identify the indices for which the inequalities in \eqref{def:maj_dom} are strict.

    \begin{definition}\label{def:slack_index}
        Given $\asvector{\mu}\preceq_w \asvector{\lambda}$, the {\em leading slack index} is the smallest number $j$ such that
        \begin{equation*}
            \sum_{i=1}^\ell \mu_{[i]} < \sum_{i=1}^\ell \lambda_{[i]}, \quad \forall \ell \in \{ j,\ldots,n\}.
        \end{equation*}
    \end{definition}
    
     This index has the property that $\mu_{[j]}$ (and no larger entry of $\asvector{\mu}$) can be increased by a nonzero amount without violating the majorization condition. 

     In deriving their main results, Roosta-Khorasani and Sz\'{e}kely make use of the following lemma \cite[12.5a]{pecaric1992convex}.
     \begin{lemma}\label{lemma:transform_spsd}
         If $\asvector{\mu}\preceq \asvector{\lambda}$, then there exists a finite sequence of vectors
         \begin{equation*}
             \asvector{\mu}\preceq \asvector{\eta}_1\preceq \ldots \preceq \asvector{\eta}_r = \asvector{\lambda}
         \end{equation*}
         so that $\asvector{\eta}_{i}$ and $\asvector{\eta}_{i+1}$ differ in two coordinates only for $i = 1,\ldots, r-1$.
     \end{lemma}
     This lemma lets us assume without loss of generality that $\asvector{\mu}$ and $\asvector{\lambda}$ differ in two coordinates only, satisfying
     \begin{equation*}
         0 \leq \lambda_{k} < \mu_{k} \leq \mu_j < \lambda_j.
     \end{equation*}

     \subsection{Indefinite majorization}

     For indefinite matrices with fixed Frobenius norm as in $\aabs(\lambda, \phi)$, the majorization condition is somewhat more complicated. Here is why: a Gamma random variable is nonnegative, so its lower tail can't get too far from the mean. The average of many Gamma random variables, however, looks more like a normal distribution which has tails in both directions. This suggests that $\tr_m^G(\mat{A})$ will have worse absolute upper tail bounds when the nonnegative eigenvalues are highly skewed and when the {\em negative} eigenvalues are tightly {\em clustered}.
     
     We will also see that $\tr_m^G(|\mat{A}|)$ has worse upper tail bounds than $\tr_m^G(\mat{A})$, for indefinite $\mat{A}$. Analogous results for lower tail bounds follow by symmetry.

    Now we describe the majorization condition more precisely.

    \begin{definition}
        For a vector $\asvector{\lambda}$, define
        \begin{equation*}
            \asvector{\lambda}^{-} := \min(\asvector{\lambda}, \asvector{0})
            \quad\text{and}\quad
            \asvector{\lambda}^{+} := \max(\asvector{\lambda}, \asvector{0}),
        \end{equation*}
        where the min and max operations are done elementwise.
    \end{definition}

    \begin{definition}\label{defn:abs_majorization}
        For vectors $\asvector{\lambda}$ and $\asvector{\mu}$, we say that $\asvector{\mu} \precabs \asvector{\lambda}$ if
        \begin{subequations}
            \begin{align}
                (\asvector{\mu}^{+})^2 &\preceq_w (\asvector{\lambda}^+)^2, \label{def:mm_pos} \\
                (\asvector{\lambda}^{-})^2 &\preceq_w (\asvector{\mu}^-)^2, \label{def:mm_neg} \\
                \sum_{i=1}^n \mu_i^2 &= \sum_{i=1}^n \lambda_i^2. \label{def:mm_sum}
            \end{align}
        \end{subequations}
    \end{definition}

    Condition \eqref{def:mm_pos} specifies that the positive entries of $\asvector{\lambda}$ are more skewed, and perhaps have more total weight, than those of $\asvector{\mu}$. Condition \eqref{def:mm_neg} specifies the opposite for the negative entries. Condition \eqref{def:mm_sum} means that the matrices associated with $\asvector{\lambda}$ and $\asvector{\mu}$ have the same Frobenius norm. This last condition also means that the associated trace estimators have the same variance.

    For the main result, we propose a lemma analogous to Lemma \ref{lemma:transform_spsd} (see Appendix \ref{sec:appendix}) for the proof).

    \begin{lemma}\label{lemma:transform}
        If $\asvector{\mu} \precabs \asvector{\lambda}$, then there exists a finite sequence of vectors
        \begin{equation*}
            \asvector{\mu}\precabs \asvector{\eta}_1\precabs \ldots \precabs \asvector{\eta}_r = \asvector{\lambda}
        \end{equation*}
        so that $\asvector{\eta}_{i}$ and $\asvector{\eta}_{i+1}$ differ in two coordinates only for $i = 1,\ldots, r-1$. Furthermore, the difference between consecutive vectors may be assumed to take one of the following forms:
        \begin{enumerate}[noitemsep]
            \item $\mu_k < \lambda_k \leq 0 \leq \mu_j < \lambda_j$;
            \item $0 \leq \lambda_k < \mu_k \leq \mu_j < \lambda_j$;
            \item $\mu_k < \lambda_k \leq \lambda_j < \mu_j \leq 0$.
        \end{enumerate}
    \end{lemma}
    Once again, we may assume without loss of generality that $\asvector{\mu}$ and $\asvector{\lambda}$ differ in two coordinates only, this time satisfying $\lambda_j^2 - \mu_j^2 = \mu_k^2 - \lambda_k^2 > 0$.

    \section{Gamma random variables}
    In this section we cover some relevant properties of Gamma random variables, and explain how to reformulate the original trace estimation problem in terms of such variables.
    
    Since real symmetric matrices are orthogonally diagonalizable and Gaussian vectors are rotationally invariant, the Gaussian trace estimator \eqref{def:G} can be written in the form
    \begin{equation}\label{eqn:chisq}
        \tr_m^G(\mat{A}) = \sum_{i=1}^n\lambda_iX_i, \quad X_i \simiid \tfrac{1}{m}\chi_m^2,
    \end{equation}
    where $\chi_m^2$ is a chi-squared variable with $m$ degrees of freedom.

    The chi-squared distribution is a special case of the Gamma distribution, and so we can generalize \eqref{eqn:chisq} as follows: 
    \begin{definition}
        Given a vector $\asvector{\lambda} = (\lambda_1,\ldots,\lambda_n)\in \mathbb{R}^n$, denote
    \begin{equation*}
        Q(\asvector{\lambda};\alpha,\beta) := \sum_{i=1}^n\lambda_iX_i, \quad X_i\simiid Gamma\left(\alpha, \beta\right).
    \end{equation*} 
    When $\alpha$ and $\beta$ are clear from context, we will use $\qlam$ for short.
    \end{definition}

    With the notation above, $\tr_m^G(\mat{A})$ has the same distribution as $Q\left(\asvector{\lambda}_{\mat{A}}; \tfrac{m}{2},\tfrac{m}{2}\right)$. 

    \subsection{Basic properties}\label{sec:gamma_properties}

    Here are a few other facts about the Gamma distribution:
    \begin{itemize}
        \item A Gamma random variable with shape $\alpha > 0$ and rate $\beta > 0$ has the PDF
        \begin{equation}\label{eqn:gammaPDF}
    		f(x) = \begin{cases}
    			\frac{\beta^\alpha}{\Gamma(\alpha)}x^{\alpha - 1}e^{-\beta x} & x \geq 0, \\
    			0 & x \leq 0.
    		\end{cases}
    	\end{equation}
        The tail decays more slowly than the tail of a normal distribution.
        \item If $X\sim Gamma(\alpha,\beta)$ then $\mathbb{E}[X] = \alpha/\beta$ and $\var[X] = \alpha/\beta^2$. Furthermore, $X$ is unimodal with mode $(\alpha-1)/\beta$ if $\alpha \geq 1$ and 0 otherwise.
        \item Gamma random variables follow a scaling law: if ${X\sim Gamma(\alpha,\beta)}$ and $\lambda > 0$, then $\lambda X\sim Gamma(\alpha, \beta/\lambda)$.
        \item Gamma random variables with the same rate parameter have an additive property: if $X_1 \sim Gamma(\alpha_1, \beta)$ and $X_2\sim Gamma(\alpha_2, \beta)$, then ${X_1+X_2 \sim Gamma(\alpha_1+\alpha_2,\beta)}$. 
        \item Conversely, Gamma random variables are {\em infinitely divisible}: if $X\sim Gamma(\alpha,\beta)$ then for any positive integer $T$, $X$ has the same distribution as $\sum_{i=1}^T X_i$, where $X_1,\ldots,X_T\simiid Gamma(\alpha/T, \beta)$.
    \end{itemize}

    \subsection{Generalizing the distribution families}
    In order to define $\qrel$ and $\qabs$ as extensions of $\arel$ and $\aabs$, we first generalize the notions of 2-norm and effective rank.

    \begin{definition}\label{lemma:scale}
        For a random variable $Q_{\asvector{\lambda}} = Q(\asvector{\lambda};\alpha,\beta)$, we define the {\em scale} of $Q_{\asvector{\lambda}}$ as 
        \begin{equation*}
            \scale(Q_{\asvector{\lambda}}) = \frac{\max_i|\lambda_i|}{\beta}.
        \end{equation*} 
    \end{definition}
    It is worth mentioning that $\scale(\qlam)$ is a property of the {\em distribution itself}, not just its representation in terms of $(\asvector{\lambda}, \alpha, \beta)$.

    \begin{definition}
        For a random variable $Q_{\asvector{\lambda}} = Q(\asvector{\lambda};\alpha,\beta)$ with nonnegative weights $\asvector{\lambda}$, we define the {\em effective shape} of $Q_{\asvector{\lambda}}$ as 
        \begin{equation*}
            \effshape(Q_{\asvector{\lambda}}) := \frac{\mathbb{E}[\qlam]}{\scale(\qlam)} = \frac{\alpha\sum_{i=1}^n\lambda_i}{\max_i \lambda_i}.
        \end{equation*}
    \end{definition}

    Now we can extend the definition of $\arel(\mu)$:

    \begin{definition}\label{def:qrel}
        For fixed $\alpha>0$ and $\beta >0$, define
        \begin{equation*}
            \mathcal{Q}_{\mathrm{rel}}(\mu; \alpha, \beta) := \{Q_{\asvector{\lambda}}\,:\,\scale(\qlam)\leq \tfrac{1}{\mu},\, \mathbb{E}[\qlam] = 1\},\quad \alpha \leq \mu,
        \end{equation*}
        where the weights $\asvector{\lambda}\in \mathbb{R}^n$ are nonnegative.
    \end{definition}

    The parameter $\mu$ in $\mathcal{Q}_{\mathrm{rel}}(\mu; \alpha, \beta)$ is a lower bound on the effective shape of $\qlam$. This definition generalizes the case of matrix trace estimation since
    \begin{equation*}
        \tr_m^G\left(\mathcal{A}_{\mathrm{rel}}(\mu)\right) = \mathcal{Q}_{\mathrm{rel}}\left(\tfrac{m\mu}{2}; \tfrac{m}{2}, \tfrac{m}{2}\right).
    \end{equation*}

    The definition of $\aabs(\lambda, \phi)$ can be extended similarly:
    \begin{definition}\label{def:qabs}
        For fixed $\alpha >0$ and $\beta>0$, define
        \begin{equation*}
            \qabs(\lambda,\phi; \alpha, \beta) := \{\qlam\,:\,\scale(\qlam) \leq \lambda, \var[\qlam] = \phi^2 \},\quad 0 < \lambda \leq \tfrac{1}{\sqrt{\alpha}}\phi.
        \end{equation*}
    \end{definition}
    In this case, we have
    \begin{equation*}
        \tr_m^G(\aabs(\lambda, \phi)) = \qabs\left(\tfrac{2\lambda}{m}, \phi\sqrt{\tfrac{2}{m}}; \tfrac{m}{2}, \tfrac{m}{2}\right).
    \end{equation*}

    \subsection{Infinite division} \label{sec:infinite_division}

    Our strategy for relaxing the original trace estimation problem is to use the fact that Gamma random variables are infinitely divisible (see Section \ref{sec:gamma_properties}). Any Gamma random variable may be rewritten as a sum of arbitrarily many Gamma random variables with smaller shape parameters $\alpha$, and this fact enables us to nest some distribution families within others.

    \begin{lemma}
        For any integer $T\geq 1$, it holds that $Q(\asvector{\lambda}; \alpha, \beta)$ has the same distribution as $Q\big((\underbrace{\asvector{\lambda},\ldots, \asvector{\lambda}}_{T\text{ times}}); \tfrac{\alpha}{T}, \beta\big)$, and consequently that 
        \begin{equation*}
            \qrel(\mu; \alpha, \beta) \subseteq \qrel\left(\mu; \tfrac{\alpha}{T}, \beta\right)
            \quad\text{and}\quad
            \qabs(\lambda, \phi; \alpha, \beta) \subseteq \qabs\left(\lambda, \phi; \tfrac{\alpha}{T}, \beta\right).
        \end{equation*}
    \end{lemma}

    By considering the limit as $T\rightarrow \infty$, we can effectively do away with the constraint of sharing a single scale parameter, and in doing so obtain the more general sets $\qrel(\mu)$ and $\qabs(\lambda, \phi)$ promised in \eqref{eqn:extrema_qrel} and \eqref{eqn:extrema_qabs}.

    \begin{definition}\label{def:qrel_generalized}
        For $\mu > 0$, let $\qrel(\mu)$ be the set of all finite linear combinations of Gamma random variables
        \begin{equation*}
            Q = \sum_{i=1}^n\lambda_iX_i,\quad X_i\simiid Gamma(\alpha_i,\beta_i),\, \lambda_i \geq 0,\, \alpha_i > 0,\, \beta_i > 0,
        \end{equation*}
        subject to the constraints $\scale[Q] \leq \tfrac{1}{\mu}$ and $\mathbb{E}[Q] = 1$. 
    \end{definition}
    \begin{definition}\label{def:qabs_generalized}
        For $\lambda >0$ and $\phi >0$, let $\qabs(\lambda, \phi)$ be the set of all finite linear combinations of Gamma random variables 
        \begin{equation*}
            Q = \sum_{i=1}^n\lambda_iX_i,\quad X_i\simiid Gamma(\alpha_i,\beta_i),\, \alpha_i > 0,\, \beta_i > 0,
        \end{equation*}
        subject to the constraints $\scale[Q] \leq \lambda$ and $\var[Q] = \phi^2$. 
    \end{definition}

    For any fixed $\alpha>0$ and $\beta>0$, for sufficiently large $T$ the set $\qrel(\mu; \tfrac{\alpha}{T}, \beta)$ contains distributions arbitrarily close to any given distribution in $\qrel(\mu)$; the same goes for $\qabs(\lambda, \phi; \tfrac{\alpha}{T}, \beta)$ and $\qabs(\lambda, \phi)$. Thus any bounds we derive for the former set will also apply to the latter. 

    \subsection{Fixed shape and scale parameters}
    The main results assume that parameters $\alpha$ and $\beta$ are fixed in order to keep the majorization definitions as simple as possible. If $\alpha_i$ and $\beta_i$ were allowed to differ for each random variable $X_i$ in the mixture $Q$, we could still map $Q$ to a random variable that takes on values $\lambda_i/\beta_i$ with probability $\alpha_i$, then make comparisons using the stochastic ordering. I don't feel that the added complexity is worth it, so won't pursue this approach further.

    \section{Majorization theorems}

    This section presents the main majorization results. The relative error result is essentially a restatement of results from \cite{roosta2015schur}; to the best of my knowledge the absolute error result is novel.

    \subsection{Relative error}

    As mentioned above, this first result was essentially proved as part of \cite[Thm.~1]{roosta2015schur}. I've still included the proof (Appendix \ref{sec:appendix:relative_error}) for the sake of completeness.

    \begin{theorem}\label{thm:main:rel}
        Fix $\mu\geq \alpha > 0$, and $\beta > 0$. Define $\xupper$ and $\xlower$ to be the extreme values of the mode of the PDF of the distribution
        \begin{equation*}
            Q_{\asvector{\lambda}}+\lambda_j\psi + \lambda_k\psi', \quad Q_{\asvector{\lambda}}\in \mathcal{Q}_{\mathrm{rel}}(\mu;\alpha,\beta),\ j\neq k,
        \end{equation*}
        where $\psi, \psi'\simiid Gamma\left(1,\beta\right)$ are exponential \rv's independent of $Q_{\asvector{\lambda}}$. If $Q_{\asvector{\lambda}}, Q_{\asvector{\mu}} \in \mathcal{Q}_{\mathrm{rel}}(\mu;\alpha,\beta)$ satisfy $\asvector{\mu}\preceq \asvector{\lambda}$, then
        \begin{align*}
            \prob\left(Q_{\asvector{\mu}} \leq x\right) &\geq \prob\left(Q_{\asvector{\lambda}} \leq x\right), \quad \forall x \geq \xupper, \\
            \prob\left(Q_{\asvector{\mu}} \leq x\right) &\leq \prob\left(Q_{\asvector{\lambda}} \leq x\right), \quad \forall x \leq \xlower.
        \end{align*}
    \end{theorem}

    That is, $\qlam$ has worse upper and lower tail bounds.

    \subsubsection{Locating the tail}
    
    But what are the values of $\xupper$ and $\xlower$? In the most general case $\mu = \alpha$, \cite[Thm.~1]{roosta2015schur} shows that $\xupper = 1 + (2\alpha)^{-1}$ and $\xlower = 1-\alpha^{-1}$. These values necessarily serve as bounds whenever $\mu\geq \alpha$, but I would like to obtain stronger results for distributions with greater effective shape.
    
    Theorem \ref{thm:rel_main_Q} gives a pessimistic result: it tells us where the tails are {\em not}, rather than where they are.
    
    \begin{theorem}\label{thm:rel_main_Q}
        If $\mu > \alpha$, the points $\xupper$ and $\xlower$ as defined in Theorem \ref{thm:main:rel} satisfy
        \begin{equation*}
            \xupper \geq 1 + (\alpha \lceil \mu/\alpha\rceil)^{-1}
            \quad\text{and}\quad
            \xlower \leq 1 - (\alpha \lceil \mu/\alpha\rceil)^{-1}.
        \end{equation*}
    \end{theorem}
    \begin{proof}
        For the bound on $\xupper$, set $Q_{\asvector{\lambda}} = \sum_{i=1}^{\lceil \mu/\alpha\rceil}\lambda X_i$ with $\lambda = \tfrac{\beta}{\alpha\lceil \mu/\alpha\rceil}$. The distribution $Q_{\asvector{\lambda}} + \lambda_1\psi + \lambda_2\psi'$ is $Gamma(\alpha \lceil \mu/\alpha\rceil + 2, \lambda^{-1}\beta)$, so its mode is ${1 + (\alpha \lceil \mu/\alpha\rceil)^{-1}}$ as desired while $\mathbb{E}[\qlam] = 1$.

        For the bound on $\xlower$, take the same $Q_{\asvector{\lambda}}$, but pad it with eigenvalues $\lambda_j=\lambda_k=0$ instead.
    \end{proof}

    Unfortunately, I have not so far been able to establish bounds in the opposite direction. I will therefore leave them as a conjecture: 

    \begin{conjecture}\label{conjecture:rel_main_Q}
        The points $\xupper$ and $\xlower$ as defined in Theorem \ref{thm:main:rel} satisfy
        \begin{equation*}
            \xupper \leq 1 + \mu^{-1}\quad \text{and}\quad 
            \xlower \geq 1 - \mu^{-1}.
        \end{equation*}
    \end{conjecture}    

    \subsection{Absolute error}
    
    The relative error analysis in the previous section has two significant limitations: first, it only applies when $\mat{A}$ is SPSD. Second, $\effrank(\mat{A})$ may sometimes be significantly larger than $\stabrank(\mat{A})$, which is a better indicator of the quality of the Gaussian trace estimator since the variance of $\tr_m^G(\mat{A})$ is $\tfrac{2}{m}\|\mat{A}\|_F^2$. 

    We therefore turn to the case where $\mat{A}$ may be indefinite but bounds on $\|\mat{A}\|_F$ and $\|\mat{A}\|_2$ are known. As before, it is useful to present the results in terms of more general Gamma random variables.

    Here is the main result on the absolute error of the estimator (proof in Appendix \ref{sec:appendix:absolute_error}).

    \begin{theorem}\label{thm:majorization_absolute}
    Fix $\alpha>0$, $\beta>0$, and $\frac{\phi}{\sqrt{\alpha}} \geq \lambda > 0$. Define $\xhatupper$ to be the supremum over all inflection points of the PDF of the distribution
    \begin{equation}\label{eqn:abs:convex_family}
        \qlam + \lambda_j\psi + \lambda_k\psi' - \mathbb{E}[\qlam],\quad \qlam\in \mathcal{Q}_{\mathrm{abs}}(\lambda, \phi; \alpha, \beta),\, j\neq k,
    \end{equation}
    where $\psi, \psi'\simiid Gamma(1,\beta)$ are exponential \rv's independent of $\qlam$. If $\qlam, Q_{\asvector{\mu}}\in \qabs(\lambda, \phi; \alpha, \beta)$ satisfy $\asvector{\mu}\precabs\asvector{\lambda}$, then  
    \begin{equation*}
        \mathrm{Pr}(Q_{\asvector{\mu}} - \mathbb{E}[Q_{\asvector{\mu}}]\leq x) \geq \mathrm{Pr}(\qlam - \mathbb{E}[\qlam]\leq x), \quad \forall |x| > \xhatupper.
    \end{equation*}
        \end{theorem} 

    This single equation (note the use of the absolute value $|x|$ in the condition) means that $\qlam$ has worse upper tail bounds {\it and} that $\qmu$ has worse lower tail bounds.

    \subsubsection{Locating the tail}

    Theorem \ref{thm:majorization_absolute} implies that the ``tails'' are the regions where all density functions of the form \eqref{eqn:abs:convex_family} are convex. For reference, a normal distribution has inflection points equal to the mean plus or minus one standard deviation. 
    
    How much further away could the tails begin? Theorem \ref{thm:abs:possible_worst_case} gives a pessimistic result.
    
    \begin{theorem}\label{thm:abs:possible_worst_case}
        The point $\xhatupper$ as defined in Theorem \ref{thm:majorization_absolute} satisfies
        \begin{equation*}
            \xhatupper \geq \phi \frac{1 + \sqrt{r\alpha + 1}}{\sqrt{r\alpha}},\quad \text{where}\ r = \left\lceil \frac{\phi^2}{\lambda^2\alpha}\right\rceil.
        \end{equation*}
    \end{theorem}

    See Appendix \ref{apx:abs_tail} for the proof. I'll leave an upper bound as a conjecture.

    \begin{conjecture}\label{conj:abs:possible_worst_case}
        The point $\xhatupper$ as defined in Theorem \ref{thm:majorization_absolute} satisfies 
        \begin{equation*}
            \xhatupper \leq \lambda + \sqrt{\phi^2+\lambda^2}.
        \end{equation*}
    \end{conjecture}
    
    \section{Application to trace estimation}

    In this section I'll rephrase the main theorems (Theorems \ref{thm:main:rel} and \ref{thm:majorization_absolute}) more directly in terms of Gaussian trace estimation.

    Here is the basic idea: the matrix $\mat{A}_{\mathrm{rel}}(\mu)$ from \eqref{def:worst_case_A_rel} (resp.~$\mat{A}_{\mathrm{abs}}(\lambda, \phi)$ from \eqref{def:worst_case_A_abs}) majorizes every other element of $\arel(\mu)$ (resp.~$\aabs(\lambda,\phi)$). Thus by the main theorems, these two matrices have the worst tail bounds. Employing the infinite division strategy of Section \eqref{sec:infinite_division} shows that the Gaussian trace estimators for these matrices are in turn tail-bounded by the Gamma distributions in \eqref{eqn:extrema_qrel} and \eqref{eqn:extrema_qabs}, respectively. Finally, increasing the sampling number $m$ will make the tails take up a larger portion of the distribution, increasing the domain over which the bounds in this paper are effective.

    First up is the relative error bound. Note that a little unit conversion is used here: if $\effrank(\mat{A}) = \mu$, then $\effshape(\tr_m^G(\mat{A})) = \frac{m\mu}{2}$.  

    \begin{theorem}\label{thm:matrix:rel}
        For nonzero SPSD $\mat{A}\in \mathbb{R}^{n\times n}$ and sampling number $m$, let $\mu = \effrank(\mat{A})$. Then there exists $\varepsilon_{\mathrm{rel}}$ such that for $\varepsilon \geq \varepsilon_{\mathrm{rel}}$, 
        \begin{align*}
            \prob\left(|\tr_m^G(\mat{A}) - \tr(\mat{A})| \geq \varepsilon\cdot \tr(\mat{A})\right) &\leq \prob\left(|\tr_m^G(\mat{A}_{\mathrm{rel}}) - 1| \geq \varepsilon\right) \\
            &\leq \prob(|X-1| \geq \varepsilon),
        \end{align*}
        where $\mat{A}_{\mathrm{rel}} = \mat{A}_{\mathrm{rel}}(\mu)$ is defined in \eqref{def:worst_case_A_rel} and $X\sim Gamma\left(\frac{m\mu}{2}, \frac{m\mu}{2} \right)$.
    \end{theorem}

    \begin{conjecture}
        $\varepsilon_{\mathrm{rel}} \leq 2/(m\mu)$. 
    \end{conjecture}

    For comparison, Corollary \ref{cor:cort_kress_rel} holds for all $\varepsilon > 0$, but with marginally weaker bounds.

    Next is the absolute error bound. Again there is a bit of unit conversion: if $\|\mat{A}\|_2 = \lambda$ and $\|\mat{A}\|_F = \phi$, then $\scale(\tr_m^G(\mat{A})) = \frac{2\lambda}{m}$ and $\var[\tr_m^G(\mat{A})] = \frac{2}{m}\phi^2$. 

    \begin{theorem}\label{thm:matrix:abs}
        Let $\mat{A}\in \mathbb{R}^{n\times n}$ be symmetric with $\|\mat{A}\|_2 = \lambda$ and $\|\mat{A}\|_F = \phi$. For fixed sampling number $m$, there exists $\varepsilon_{\mathrm{abs}}$ such that for $\varepsilon \geq \varepsilon_{\mathrm{abs}}$, 
        \begin{align*}
            \prob\left(|\tr_m^G(\mat{A}) - \tr(\mat{A})| \geq \varepsilon\right) &\leq  2\,\prob\left(\tr_m^G(\mat{A}_{\mathrm{abs}}) - \tr(\mat{A}_{\mathrm{abs}}) \geq \varepsilon\right) \\
            &\leq 2\,\prob(X -\mathbb{E}[X] \geq \varepsilon),
        \end{align*}
        where $\mat{A}_{\mathrm{abs}} = \mat{A}_{\mathrm{abs}}(\lambda, \phi)$ is defined in \eqref{def:worst_case_A_abs} and $X\sim Gamma\left(\frac{m\rho}{2}, \frac{m}{2\lambda} \right)$ with $\rho = \stabrank(\mat{A}) = \phi^2/\lambda^2$.
    \end{theorem}

    \begin{conjecture}\label{conj:matrix:abs}
        $\varepsilon_{\mathrm{abs}} \leq \frac{2\lambda}{m} + \sqrt{\frac{2}{m}\phi^2 + \left(\frac{2\lambda}{m}\right)^2}$.
    \end{conjecture}

    For comparison, Theorem \ref{thm:cort_kress_abs} holds for all $\varepsilon > 0$, but with marginally weaker bounds. 
    
    The factor of 2 appearing in the bounds in Theorem \ref{thm:matrix:abs} is due to the fact that the result uses $\mat{A}_{\mathrm{abs}}$ for the upper tail bound and $-\mat{A}_{\mathrm{abs}}$ for the lower tail bound (resp.~$X$ and $-X$). Note also that in Conjecture \ref{conj:matrix:abs} the terms $\phi$ and $\lambda$ scale differently as the sampling number $m$ increases. If true, the conjecture implies that $\varepsilon_{\mathrm{abs}}$ converges to $\phi$ as $m\rightarrow \infty$.

    \section{Conclusion}

    So what exactly are the practical implications of all of this? One takeaway, following from Theorem \ref{thm:matrix:rel}, is that the relative error bound for SPSD matrices depends on $m\cdot \effrank(\mat{A})$, so having a large effective rank is just as beneficial as using a large sampling number. One particular application is in estimating the Frobenius norm of a matrix. The worst-case bounds of \cite{roosta2015gamma} hold when the matrix has rank one, but with a lower bound on the stable rank of $\mat{A}$ these bounds can be tightened. This could in turn reduce the number of samples required to estimate the Frobenius norm to a given tolerance, as in \cite[Lemma 2.2]{persson2023hutch}.

    Another takeaway is that the bounds by Cortinovis and Kressner in Theorem \ref{thm:cort_kress_abs} and Corollary \ref{cor:cort_kress_rel} are already quite good. The improvements made in this paper, which establishes bounds that are tight given certain information about $\mat{A}$, are fairly minor. If you want any further improvements to these tail bounds, you will have to use more information about the spectrum of $\mat{A}$.

    In practice, you probably shouldn't use the Gaussian trace estimator on its own. If your matrix has a small stable rank or effective rank, you'd do better to combine the trace estimator with a low-rank approximation such as in \cite{meyer2021hutch, persson2023hutch, epperly2024xtrace}. Furthermore, you can reduce the variance by using the Hutchinson estimator (sampling vectors with random $\pm 1$ entries), or by normalizing the Gaussian vectors to have unit length \cite{epperly2024trace}. The main proof techniques in this paper do not work on the Hutchinson estimator because it has a discrete distribution. The techniques could potentially be applied to the normalized Gaussian estimator, but the proofs will be more complicated. 

    Finally, this paper successfully applies the proof techniques of \cite{szekely2003gaussian, roosta2015gamma, roosta2015schur} to obtain a majorization result and tight tail bounds for matrices with fixed $2$-norm and Frobenius norm. The application to trace estimation may not yield results significantly better than those that can be obtained through concentration inequalities, but the approach\textemdash comparing the CDFs of distributions to solve an optimization problem directly\textemdash is markedly different.

	\appendix
        \section{Proofs} \label{sec:appendix}
        \begin{proof}[Proof of Lemma \ref{lemma:transform}]
            Recall from Definition \ref{defn:abs_majorization} that $\asvector{\mu} \precabs \asvector{\lambda}$ if 
            \begin{align*}
                (\asvector{\mu}^{+})^2 \preceq_w (\asvector{\lambda}^+)^2, \tag{\ref{def:mm_pos} revisited}\\
                (\asvector{\lambda}^{-})^2 \preceq_w (\asvector{\mu}^-)^2,\tag{\ref{def:mm_neg} revisited}\\
                \sum_{i=1}^n \mu_i^2 = \sum_{i=1}^n \lambda_i^2. \tag{\ref{def:mm_sum} revisited}
            \end{align*}

            First suppose that the inequalities in \eqref{def:mm_pos} and \eqref{def:mm_neg} are strict (either both are, or neither). Find the leading slack indices (see Definition \ref{def:slack_index}) $j^+$ and $j^-$ for each inequality. Increase the corresponding nonnegative coordinate\footnote{If no such coordinate exists, pad $\asvector{\mu}$ with zeros as needed. For example, if  $\asvector{\lambda} = (2,2)$ and $\asvector{\mu} = (-2, -2)$, we would get the sequence $(-2,-2,0)\mapsto (-2, 0, 2)\mapsto (0, 2, 2)$. Such padding allows us to keep the sum of squares constant while changing the vectors continuously.} of $\asvector{\mu}$ while shrinking the negative coordinate toward zero, keeping their sum-of-squares constant, until one of the slack inequalities becomes an equality. Repeat a finite number of times to eliminate all slack.

            Then, apply Lemma \eqref{lemma:transform_spsd} to the cases $(\asvector{\mu}^+)^2 \preceq (\asvector{\lambda}^+)^2$ and $(\asvector{\lambda}^-)^2 \preceq (\asvector{\mu}^-)^2$ separately.
        \end{proof}

        \subsection{Relative error majorization theorem} \label{sec:appendix:relative_error}

        As a reminder, most of this proof is substantially the same as the one given in \cite{roosta2015schur}. I provide it here in part to show how the proof techniques relate to those used for the absolute error case. 

        \begin{proof}[Proof of Theorem \ref{thm:main:rel}]
        Lemma \ref{lemma:transform_spsd} implies that we can without loss of generality assume that $\asvector{\mu}$ and $\asvector{\lambda}$ differ in two coordinates only, satisfying $0\leq \lambda_k < \mu_k \leq \mu_j < \lambda_j$. For $t\in [0,1]$, define
        \begin{align*}
            \nu_i(t) &:= t\lambda_i + (1-t)\mu_i, \quad i \in \{j,k\}, \\
            \nu_i(t) &:= \lambda_i, \quad i\neq j,\\
            Y(t) &:= \sum_{i=1}^n \nu_i(t)X_i.
        \end{align*}
        This interpolation satisfies $Y(0) = \qmu$ and $Y(1) = \qlam$. Our goal is to show that for certain fixed values of $x$, the CDF $F_{Y(t)}(x)$ is monotonic for $t\in [0,1]$. 

        Take the Laplace transform of $F_{Y(t)}$ as
        \begin{align*}
        \begin{split} 
            J(t,z) := \mathcal{L}[F_{Y(t)}](z) &= \int_0^\infty e^{-zx}F_{Y(t)}(x)\,dx\\
            &= \frac{-1}{z}\int_0^\infty F_{Y(t)}(x)\,d\left(e^{-zx}\right)\\
            &= \frac{1}{z}\int_0^\infty e^{-zx}\,dF_{Y(t)}(x)\\
            &= \frac{1}{z}\mathcal{L}[Y(t)](z),
        \end{split}        
        \end{align*}
        where $\mathcal{L}[Y(t)](z):= \mathbb{E}[e^{-zY(t)}]$, the Laplace transform of $Y(t)$, satisfies
        \begin{equation}\label{eqn:laplace:rel}
            \mathcal{L}[Y(t)](z) = \prod_{i=1}^n\left(1 + \frac{\nu_i(t)z}{\beta}\right)^{-\alpha},\quad z \in \mathbb{C},\ \Re(z) > -\min_{1\leq i \leq n}\frac{\beta}{\nu_i(t)}.
        \end{equation}
        Differentiating with respect to $t$ yields
        \begin{align*}
            \frac{\partial }{\partial t}J(t,z) &= J(t,z)\frac{\partial}{\partial t}\ln J(t,z) \\
            &= J(t,z) \frac{\partial}{\partial t}\sum_{i\in \{j,k\}}\left(-\alpha \ln\left(1 + \frac{\nu_i(t)z}{\beta}\right) \right)\\
            &= -J(t,z)z\frac{\alpha}{\beta}\sum_{i\in \{j,k\}}\frac{\lambda_i-\mu_i}{1 + \frac{\nu_i(t)z}{\beta}}.
        \end{align*}
        Recall that $\lambda_j - \mu_j = \mu_k - \lambda_k > 0$ (since $\asvector{\mu}\preceq \asvector{\lambda}$), and see that 
        \begin{equation}\label{eqn:fraction_sum}
            \frac{1}{1 + \frac{\nu_j(t)z}{\beta}} - \frac{1}{1 + \frac{\nu_k(t)z}{\beta}} = (\nu_k(t)-\nu_j(t))\frac{z}{\beta}\prod_{i\in \{j,k\}} \left(1 + \frac{\nu_i(t)z}{\beta}\right)^{-1}.
        \end{equation}
        It follows that
        \begin{equation*}
            \frac{\partial J}{\partial t}(t,z) = J(t,z)z^2\frac{\alpha}{\beta^2} \frac{(\lambda_j-\mu_j)(\nu_j(t)-\nu_k(t))}{\prod_{i\in \{j,k\}} \left(1 + \frac{\nu_i(t)z}{\beta}\right)}.
        \end{equation*}
        Substitute $J(t,z) := \mathcal{L}[F_{Y(t)}](z)$ and apply the inverse Laplace transform to get        
        \begin{align*}
            \frac{\partial}{\partial t}F_{Y(t)}(x) &= \frac{\alpha}{\beta^2}(\lambda_j-\mu_j)(\nu_j(t)-\nu_k(t))\frac{\partial^2}{\partial x^2}\prob\left(Y(t) + \nu_j(t)\psi + \nu_k\psi'  \leq x\right)\\
            &= \frac{\alpha}{\beta^2}(\lambda_j-\mu_j)(\nu_j(t)-\nu_k(t))\frac{\partial}{\partial x}f_{Y(t) + \nu_j(t)\psi + \nu_k\psi'}(x),
        \end{align*}
        where $\psi, \psi' \simiid Gamma(1,\beta)$ are \iid~exponential \rv's which are also independent of $Y(t)$, and $f(x)$ is the probability density function (PDF). Now $\lambda_j > \mu_j$ by assumption, and for any $t\in [0,1]$ it holds that $\nu_j(t) \geq \nu_k(t)$ as well. Thus for each $t\in [0,1]$ the left-hand side $\frac{\partial}{\partial t}F_{Y(t)}(x)$ has the same sign as $\frac{\partial}{\partial x}f_{Y(t) + \nu_j(t)\psi + \nu_k\psi'}(x)$, the derivative of the density function. 

        It is known that the distribution of any linear combination of Gamma random variables is unimodal (see \cite[Thm.~4]{szekely2003gaussian}), so by the definition of $\xupper$ and $\xlower$, the density function is increasing on $(0,\xlower)$ and decreasing on $(\xupper, \infty)$. Therefore, for any $x$ in either of these regions, $F_{Y(t)}(x)$ is monotonic with respect to $t\in [0,1]$. Since $Y(0) = \qmu$ and $Y(1) = \qlam$, the desired inequalities follow.
        \end{proof}

        \subsection{Absolute error majorization theorem} \label{sec:appendix:absolute_error}

        \begin{proof}[Proof of Theorem \ref{thm:majorization_absolute}]
             Lemma \ref{lemma:transform} implies that we can without loss of generality assume that $\asvector{\lambda}$ and $\asvector{\mu}$ differ in two coordinates, satisfying $\lambda_j^2 - \mu_j^2 = \mu_k^2 - \lambda_k^2 > 0$. For $t\in [0,1]$, define
		\begin{align*}
			\begin{split}
				\nu_i(t) &:= \mathrm{sgn}(\lambda_i+\mu_i)\sqrt{t\lambda_i^2 + (1-t)\mu_i^2},\quad i \in \{j,k\},\\
				\nu_i(t) &:= \lambda_i, \quad i \neq j,\,k,\\
				Y(t) &:= \sum_{i=1}^n \nu_i(t)(X_i-\alpha/\beta).
			\end{split}
		\end{align*}
        Note that $\mathbb{E}[Y(t)] = 0$ by design. As for the term $\mathrm{sgn}(\lambda_i+\mu_i)$, Lemma \ref{lemma:transform} implies that we need only consider cases where $\lambda_i$ and $\mu_i$ are both nonpositive or both nonnegative. This slightly awkward term therefore ensures that $Y(0) = \qmu$ and $Y(1) = \qlam$. We also note that for $i \in \{j,k\}$ and $\nu_i(t)\neq 0$, we have 
		\begin{equation}\label{eqn:nu_prime}
			\ddt\nu_i(t) = \frac{\lambda_i^2-\mu_i^2}{2\nu_i(t)}.
		\end{equation}
        Our goal is again to show that for certain fixed values of $x$, the CDF $F_{Y(t)}$ is monotonic for $t\in [0,1]$.
        
        Take the Laplace transform of $F_{Y(t)}$ as
		\begin{equation*}
			J(t,q) := \mathcal{L}[F_{Y(t)}](z) = \frac{1}{z}\mathcal{L}[Y(t)](z),
		\end{equation*}
        where $\mathcal{L}[Y(t)](z):= \mathbb{E}[e^{-zY(t)}]$, the Laplace transform of $Y(t)$, satisfies
		\begin{equation*}
			\mathcal{L}[Y(t)](z) = \prod_{i=1}^ne^{z\nu_i(t)\alpha/\beta}\left(1 + \frac{\nu_i(t)z}{\beta}\right)^{-\alpha},
		\end{equation*}
		defined over the strip
		\begin{equation}\label{eqn:strip:abs}
			z\in \mathbb{C}, \quad  \frac{-\beta}{\max_{\nu_i(t) > 0}\nu_i(t)} < \Re(z) < \frac{\beta}{\max_{\nu_i(t) < 0}|\nu_i(t)|}.
		\end{equation}
        As compared with \eqref{eqn:laplace:rel}, the extra term $e^{z\nu_i(t)\alpha/\beta}$ comes from our using the centered variables $X_i-\alpha/\beta$. The region of convergence in \eqref{eqn:strip:abs} is a strip rather than a half-plane because $Y(t)$ may include both positive and negative weights $\nu_i(t)$.

		Differentiating $J(t,z)$ with respect to $t$ and using \eqref{eqn:nu_prime} yields
		\begin{align*}
			\frac{\partial }{\partial t}J(t, z) &= J(t,z)\frac{\partial}{\partial t}\ln J(t,z) \\
					&= J(t,z)\sum_{i\in \{j,k\}}\frac{\partial}{\partial t}\left( z\nu_i(t)\frac{\alpha}{\beta} - \alpha \ln\left(1 + \frac{\nu_i(t)z}{\beta}\right)\right)\\
					&= J(t,z)\sum_{i\in \{j,k\}}z\frac{\alpha}{\beta}\left(1 - \frac{1}{1+\tfrac{\nu_i(t)z}{\beta}}\right)\ddt\nu_i(t)\\
					&= J(t,z)z^2\frac{\alpha}{2\beta^2}\sum_{i\in \{j,k\}}\frac{\lambda_i^2-\mu_i^2}{1 + \nu_i(t)z/\beta}.
		\end{align*}
        Recall that $\lambda_j^2 - \mu_j^2 = \mu_k^2 - \lambda_k^2 > 0$, and again use \eqref{eqn:fraction_sum} to get
        \begin{equation*}
            \frac{\partial }{\partial t}J(t, z) = J(t,z)z^3\frac{\alpha}{2\beta^3}\frac{(\lambda_j^2-\mu_j^2)(\nu_k(t)-\nu_j(t))}{\left(1+\tfrac{\nu_j(t)z}{\beta}\right)\left(1+\tfrac{\nu_k(t)z}{\beta}\right)}.
        \end{equation*}
		Taking the inverse transform then gives
		\begin{align*}
			\frac{\partial}{\partial t} F_{Y(t)}(x) &= \frac{\alpha}{2\beta^3}(\lambda_j^2-\mu_j^2)(\nu_k(t)-\nu_j(t))\frac{\partial^3}{\partial x^3}\mathrm{Pr}\left(Y(t) + \nu_j(t)\psi + \nu_k(t)\psi' \leq x\right) \\
			&= \frac{\alpha}{2\beta^3}(\lambda_j^2-\mu_j^2)(\nu_k(t)-\nu_j(t))\frac{\partial^2}{\partial x^2}f_{Y(t)+\nu_j(t)\psi + \nu_k(t)\psi'}(x),
		\end{align*}
		where $\psi, \psi'\simiid Gamma(1,\beta)$ are \iid~exponential \rv's which are also independent of $Y(t)$, and $f(x)$ is the probability density function (PDF). 

        By design, $\lambda_j^2 - \mu_j^2 > 0$. Checking the three cases considered in Lemma \ref{lemma:transform} shows also that $\nu_k(t) - \nu_j(t) \leq 0$ for $t\in [0,1]$. Thus the sign of $\frac{\partial}{\partial t}F_{Y(t)}(x)$ is always opposite that of $\frac{\partial^2}{\partial x^2}f_{Y(t)+\nu_j(t)\psi_j + \nu_k(t)\psi_k}(x)$, which is the convexity of the density function. By the definition of $\xhatupper$, the density function is convex on $(\xhatupper, \infty)$, and so for any $x$ in this region, $F_{Y(t)}(x)$ decreases monotonically for $t\in [0,1]$. Since $Y(0) = \qmu - \mathbb{E}[\qmu]$ and $Y(1) = \qlam - \mathbb{E}[\qlam]$, the desired upper-tail inequality follows. 

        The lower-tail bound follows by symmetry. 		
	\end{proof}

    \subsection{Absolute error tail results} \label{apx:abs_tail}
    \begin{proof}[Proof of Theorem \ref{thm:abs:possible_worst_case}]
        Let $\qlam = \sum_{i=1}^r\hat{\lambda} X_i$, where $X_i\simiid Gamma(\alpha, \beta)$, $r = \lceil \phi^2/(\lambda^2\alpha)\rceil$ and $\hat{\lambda} = \frac{\beta\phi}{\sqrt{r\alpha}}$. This is a valid choice of distribution since 
        \begin{equation*}
            \var[\qlam] = r\hat{\lambda}^2\var[X_1] = r\left(\frac{\beta^2\phi^2}{r\alpha}\right)\frac{\alpha}{\beta^2} = \phi^2
        \end{equation*}
        and
        \begin{equation*}
            \scale(\qlam) = \frac{\hat{\lambda}}{\beta} = \frac{\phi}{\sqrt{r\alpha}} \leq \frac{\phi}{\sqrt{\phi^2/\lambda^2}} = \lambda.
        \end{equation*}
        Then $\mathbb{E}[\qlam] = r\hat{\lambda}\frac{\alpha}{\beta} = \phi\sqrt{r\alpha}$, and the \rv~$\qlam + \lambda_1\psi + \lambda_2\psi'$ has distribution $Gamma(r\alpha + 2, \beta \hat{\lambda}^{-1})$, or equivalently, $Gamma(r\alpha + 2, \sqrt{r\alpha}/\phi)$.

        Now in general, a $Gamma(\alpha, \beta)$ random variable has density proportional to $x^{\alpha - 1}e^{-\beta x}$ for $x\geq 0$ (see \eqref{eqn:gammaPDF}). The inflection points are where the second derivative is equal to zero. Since
        \begin{align*}
            \frac{\mathrm{d}^2}{\mathrm{d}x^2}\, x^{\alpha - 1}e^{-\beta x} &= \frac{\mathrm{d}}{\mathrm{d}x}\,(-\beta x^{\alpha -1} + (\alpha - 1)x^{\alpha-2})e^{-\beta x}\\
              &= (\beta^2x^{\alpha-1}-2\beta(\alpha-1)x^{\alpha-2}+(\alpha-1)(\alpha-2)x^{\alpha-3})e^{-\beta x}\\
              &= ((\beta x)^2 - 2(\alpha-1)(\beta x) + (\alpha-1)(\alpha-2))x^{\alpha-3}e^{-\beta x},
        \end{align*}
        the nontrivial inflection points are the two zeros of the quadratic term. The larger of the two zeros is given by
        \begin{equation*}
            \hat{x}_+ = \beta^{-1}\frac{2(\alpha-1) + \sqrt{4(\alpha-1)^2-4(\alpha-1)(\alpha-2)}}{2} = \frac{\alpha - 1 + \sqrt{\alpha - 1}}{\beta}.
        \end{equation*}
        Substitute for $\alpha$ and $\beta$ the values $r\alpha + 2$ and $\sqrt{r\alpha}/\phi$, and subtract $\mathbb{E}[\qlam]$ to get
        \begin{equation}\label{eqn:abs:pessimistic_bound}
            \xhatupper \geq \frac{r\alpha + 1 + \sqrt{r\alpha +1}}{\sqrt{r\alpha}/\phi} - \phi\sqrt{r\alpha} = \phi \frac{1 + \sqrt{r\alpha + 1}}{\sqrt{r\alpha}}.
        \end{equation}
    \end{proof}

    Conjecture \ref{conj:abs:possible_worst_case} is motivated by the fact that the right-hand-side of \eqref{eqn:abs:pessimistic_bound} is bounded above as 
    \begin{equation*}
        \phi \frac{1 + \sqrt{r\alpha + 1}}{\sqrt{r\alpha}}
        \leq
        \phi \frac{1 + \sqrt{\phi^2/\lambda^2 + 1}}{\sqrt{\phi^2/\lambda^2}}
        = \lambda + \sqrt{\phi^2+\lambda^2}.
    \end{equation*}

	\bibliography{references.bib}

\begin{thebibliography}{10}

\bibitem{cortinovis2021indefinite}
Alice Cortinovis and Daniel Kressner.
\newblock On randomized trace estimates for indefinite matrices with an application to determinants.
\newblock {\em Foundations of Computational Mathematics}, 22(3):875--903, 2022.

\bibitem{epperly2024trace}
Ethan~N Epperly.
\newblock Don't use gaussians in stochastic trace estimation.
\newblock \url{https://www.ethanepperly.com/index.php/2024/01/28/}, Jan 2024.
\newblock Accessed: 2024-11-01.

\bibitem{epperly2024xtrace}
Ethan~N. Epperly, Joel~A. Tropp, and Robert~J. Webber.
\newblock Xtrace: Making the most of every sample in stochastic trace estimation.
\newblock {\em SIAM Journal on Matrix Analysis and Applications}, 45(1):1--23, 2024.

\bibitem{meyer2021hutch}
Raphael~A Meyer, Cameron Musco, Christopher Musco, and David~P Woodruff.
\newblock Hutch++: Optimal stochastic trace estimation.
\newblock In {\em Symposium on Simplicity in Algorithms (SOSA)}, pages 142--155. SIAM, 2021.

\bibitem{persson2023hutch}
David Persson, Alice Cortinovis, and Daniel Kressner.
\newblock Improved variants of the hutch++ algorithm for trace estimation.
\newblock {\em SIAM Journal on Matrix Analysis and Applications}, 43(3):1162--1185, 2022.

\bibitem{persson2022improvedvariantshutchalgorithm}
David Persson, Alice Cortinovis, and Daniel Kressner.
\newblock Improved variants of the hutch++ algorithm for trace estimation (preprint), 2022.

\bibitem{pecaric1992convex}
Josip~E Pe\v{c}ari\'{c} and Yung~Liang Tong.
\newblock {\em Convex functions, partial orderings, and statistical applications}.
\newblock Academic Press, 1992.

\bibitem{roosta2015schur}
Farbod Roosta-Khorasani and G{\'a}bor~J. Sz{\'e}kely.
\newblock Schur properties of convolutions of gamma random variables.
\newblock {\em Metrika}, 78(8):997--1014, 2015.

\bibitem{roosta2015gamma}
Farbod Roosta-Khorasani, G\'{a}bor~J. Sz\'{e}kely, and Uri~M. Ascher.
\newblock Assessing stochastic algorithms for large scale nonlinear least squares problems using extremal probabilities of linear combinations of gamma random variables.
\newblock {\em SIAM/ASA Journal on Uncertainty Quantification}, 3(1):61--90, 2015.

\bibitem{szekely2003gaussian}
G{\'a}bor~J. Sz{\'e}kely and Nail~K. Bakirov.
\newblock Extremal probabilities for gaussian quadratic forms.
\newblock {\em Probability Theory and Related Fields}, 126(2):184--202, 2003.

\end{thebibliography}
	\bibliographystyle{plain}

\end{document}